\RequirePackage[l2tabu, orthodox]{nag}
\documentclass[a4paper]{amsart}

\usepackage[T1]{fontenc}
\usepackage{microtype}
\usepackage[rm={lining=true}]{cfr-lm}

\usepackage{amssymb}
\usepackage{amsfonts,amsmath}
\usepackage[foot]{amsaddr}
\usepackage{amsthm}
\usepackage{graphicx}
\usepackage{pinlabel}
\usepackage{subcaption,booktabs}

\usepackage{hyperref}
\usepackage[capitalise,noabbrev]{cleveref} 
\crefformat{equation}{(#2#1#3)}

\DeclareMathOperator{\tr}{tr}
\let\Re\relax
\DeclareMathOperator{\Re}{Re}
\let\Im\relax
\DeclareMathOperator{\Im}{Im}

\def\IH{{\mathbb{H}}}

\def\IH{{\mathbb{H}}}

\def\IR{{\mathbb{R}}}

\def\IZ{{\mathbb{Z}}}

\def\IC{\mathbb{C}}

\def\oC{\hat{\IC}}
\newcommand{\abs}[1]{\left\lvert#1\right\rvert}
\newcommand{\df}[1]{\textit{#1}}

\newtheorem{theorem}{Theorem}[section]
\newtheorem{lemma}[theorem]{Lemma}
\AddToHook{env/lemma/begin}{\crefalias{theorem}{lemma}}
\newtheorem{corollary}[theorem]{Corollary}
\AddToHook{env/corollary/begin}{\crefalias{theorem}{corollary}}
\theoremstyle{definition}
\newtheorem{definition}[theorem]{Definition}
\AddToHook{env/definition/begin}{\crefalias{theorem}{definition}}
\newtheorem{example}[theorem]{Example}
\AddToHook{env/example/begin}{\crefalias{theorem}{example}}
\theoremstyle{remark}
\newtheorem{remark}[theorem]{Remark}
\AddToHook{env/remark/begin}{\crefalias{theorem}{remark}}

\newcommand{\GL}{\mathsf{GL}}
\newcommand{\PSL}{\mathsf{PSL}}
\newcommand{\SL}{\mathsf{SL}}
\newcommand{\Z}{\mathbb{Z}}
\newcommand{\C}{\mathbb{C}}
\newcommand{\R}{\mathbb{R}}
\DeclareMathOperator{\Fix}{Fix}

\usepackage[backend=bibtex,style=numeric,autocite=inline, maxnames=8, isbn=false, sortcites]{biblatex}
\title[Bounding deformation spaces of Kleinian groups]{Bounding deformation spaces of Kleinian groups with two generators}

\author[A. Elzenaar]{Alex Elzenaar}
\address[A.E.]{School of Mathematics,  Monash University, Melbourne, Australia.}
\author[J. Gong]{Jinhua Gong}
\address[J.G.]{Department of Mathematics, UAE University, Abu Dhabi, United Arab Emirates.}
\author[G.J. Martin]{Gaven J. Martin}
\address[G.J.M.]{Institute for Advanced Study,  Massey University, Auckland, New Zealand.}
\author[J. Schillewaert]{Jeroen Schillewaert}
\address[J.S.]{Department of Mathematics, The University of Auckland, New Zealand.}
\thanks{Work of authors partially supported by the New Zealand Marsden Fund. AE was supported by an Australian Government Research Training Program (RTP) Scholarship during part of the period that this work was undertaken.}


\begin{document}

\begin{abstract}  In this article we provide simple and provable bounds on the size and shape of the locus of discrete subgroups of $\PSL(2,\IC)\cong \mathrm{Isom}^+(\IH^3)$ which split as a free product of cyclic groups $\IZ_p*\IZ_q$, $3\leq p,q \leq \infty$. These bounds are sharp and meet the highly fractal boundary of the deformation space in four cusp groups. Such bounds have great utility in computer assisted searches for extremal Kleinian groups so as to identify universal constraints (volume, length spectra, etc.) on the geometry and topology of hyperbolic $3$-orbifolds. As an application, we prove a strengthened version of a conjecture by Morier-Genoud, Ovsienko, and Veselov, motivated by the theory of quantum rational numbers, on the faithfulness of the specialised Burau representation.
\end{abstract}

\maketitle

\section{Introduction}
Let $ 3 \leq p,q \leq \infty $. The $\PSL(2,\IC) $ character variety of $ \IZ_p * \IZ_q $ is a one-complex dimensional space. Picking a particular normalisation for the representations, depending on
a single complex parameter $ \rho $, in \cref{defn:omega_bounds} we give an open region $\Omega_{p,q} \subset \C $ bounded by 12 lines such that if $ \rho \not\in \Omega_{p,q} $ then the corresponding
representation is faithful and discrete. More precisely, our main result is the following:
\begin{theorem}\label{thm:main}
  Let $X,Y\in \SL(2,\IC)$ be primitive elliptic elements of order $p$ and $q$, where $3\leq p\leq q$. Let $\rho$ be a solution of
  \begin{displaymath}
    \rho  \left( \rho-4\sin \frac{\pi}{p}\sin\frac{\pi}{q} \right) = \tr\,[X,Y]-2.
  \end{displaymath}
  If $\rho\not\in \Omega_{p,q}$, then $\langle X,Y\rangle\cong \langle X\rangle*\langle Y\rangle$ and the group is discrete.  This result is sharp at four cusp groups lying on the boundary of $\Omega_{p,q}$.
\end{theorem}
The majority of this theorem will be proved as \cref{Vpq}, with one final bound being proved in \cref{lem:main_res_part_2}.

Our methods rely on the complex analytic theory of M\"obius transformations. The interior of the locus of discrete groups isomorphic to $\IZ_p*\IZ_q$, denoted $\mathcal{R}_{p,q}$, is the quasiconformal deformation space
of the representations within it \autocite{EMS1}. As a complex manifold, it is the quotient by a $\IZ$ action (the twist group) of the Teichm\"uller space of the marked Riemann surface supported on the Riemann sphere $\oC$
with four marked points corresponding to two cone points of order $ p $ and two of order $q$; but as a subset of the character variety it has a highly complicated fractal boundary. We are able, through careful
geometric analysis, to produce bounds on this complicated subset of $ \IC $ using subtle but entirely classical Euclidean and conformal arguments. Increasingly precise
bounds on $\mathcal{R}_{\infty,\infty}$ were provided from the late 1940s \autocite{Sanov,Brenner,CJR} to the late 1960s \autocite{LU} before the connections to 3-manifold topology and knot theory were
understood. Many of these ideas, particularly those in \autocite{LU}, underpin our work here. The common thread between all of these papers is their reliance on so-called \textit{ping-pong arguments}
first introduced by Klein in order to prove that the groups within the bounds are discrete. We also use ping-pong methods, and in particular the formulation of the Klein combination theorem
given by Maskit~\cite[\S VII.A.10]{Mas}, see \cref{lem:pingpong}.

The bounds we achieve have great utility in computer assisted searches for extremal Kleinian groups so as to identify universal constraints (volume, length spectra, etc) on the geometry and topology of hyperbolic $3$-orbifolds such as those found in \autocite{M1,M2,M3,M4,M5}.   This is because the infinitely many lattices in $\PSL(2,\IC)$ generated by two elements of order $p,q$ are naturally identified with discrete groups corresponding to points in the bounded region $\IC\setminus \mathcal{R}_{p,q}$ and so determining membership in $ \mathcal{R}_{p,q} $ provides a certificate that a group is not a lattice.

As an application of our work, we conclude the paper by proving a strengthened version of a conjecture by Morier-Genoud, Ovsienko, and Veselov \autocite[Conjecture 3.2]{mgov24} on the faithfulness
of the specialised Burau representation, see \cref{cor:full_bound}. This conjecture was motivated by the theory of quantum rational numbers, and through this link the quasiconformal deformation
space of $ \PSL(2,\IZ) \simeq \Z_3 * \IZ_2 $ is related to various developments in cluster algebras, quantum calculus, knot invariants, and enumerative geometry.

In \autocite{mgov24, mgo20}, the link between the $q$-deformed modular group $ \PSL(2,\IZ)_q $ and the reduced Burau representation of $ B_3 $ is explained. They also give combinatorial
interpretations for the coefficients of elements of $ \PSL(2,\IZ)_q $ (which are polynomial in $ q $) in terms of the $q$-deformed rational numbers. One interesting interpretation which is only tangential
to the results of this paper is related to the Jones polynomial of $2$-bridge links. If $ k $ is the $ r/s $ $2$-bridge link, then construct the $ r/s$ slope-word in $\PSL(2,\IZ)_q $. Write this word as
\begin{displaymath}
  \begin{bmatrix}
  \mathcal{R}_{r/s}(q) & \mathcal{U}_{r/s}(q)\\
  \mathcal{S}_{r/s}(q) & \mathcal{V}_{r/s}(q)
  \end{bmatrix} \in \PSL(2,\IZ)_q;
\end{displaymath}es
then the Jones polynomial of $k$ is $ q\mathcal{R}_{r/s}(q) + (1-q) \mathcal{S}_{r/s}(q) $. This shows a connection between the quasiconformal deformation space $ \mathcal{R}_{p,q} $ and quantum
invariants of one of the two natural closures of braids on three strands, namely the closure which produces $2$-bridge links. There is already a strong link between $ \mathcal{R}_{\infty,\infty} $
and $2$-bridge links which goes back to work of Riley in the 1970s, which we describe in \autocite{EMS1}; there, every $2$-bridge link is assigned to a different word (which we call \df{Farey words})
in the generators $ X $ and $ Y$, and the trace polynomials (\df{Farey polynomials}) of these words control the entire structure of $ \mathcal{R}_{\infty,\infty} $. The problem of relating the
Farey polynomial of a particular $2$-bridge link (a function on $ \mathcal{R}_{\infty,\infty} $) to the Jones polynomial of the link (which is a function on $ \mathcal{R}_{2,3} $) remains open.

\section{Notation and elementary definitions}\label{tracechecks}
\begin{definition}\label{XY}
  For $\alpha,\beta,\rho\in\IC^*$ we define two matrices $A$ and $B=B_\rho$ in $\PSL(2,\IC)$ by
  \begin{displaymath}
    A = \begin{pmatrix} \alpha & 1 \\ 0 & \alpha^{-1} \end{pmatrix},\quad
    B = B_\rho = \begin{pmatrix} \beta & 0 \\ \rho & \beta^{-1} \end{pmatrix}.
  \end{displaymath}
  We will also use $A$ and $B$ to also denote the M\"obius transformations associated with these matrices.  Thus $ A(z) = \alpha^2 z + \alpha $
  and $ B(z) = \beta^2 z/(\rho\beta z + 1) $.
\end{definition}
We are mainly interested in the case that $\alpha=e^{i\pi/p}$ and $\beta=e^{i\pi/q}$ for integers $p,q\geq 3$. Often we can get away with omitting the case that either $ p $ or $ q $
is $2$, as if $B^2=I$ then $\langle A,BAB^{-1} \rangle$ has index two in $\langle A,B \rangle$ and so these groups are simultaneously discrete or not. However, the majority of our
work goes through in the case that $ p \geq 3 $ and $ q \geq 2 $.

The fixed points of $A$ and $B$ in the Riemann sphere $\oC $ are
\begin{align*}
  \Fix A &= \left\{\infty,\frac{\alpha }{1-\alpha ^2 } \right\} \quad \left[ {} =\left\{\infty,\frac{i}{2} \csc  \frac{\pi }{p}  \right\} \; \text{if $\alpha=e^{i\pi/p}$}\right] \\
  \Fix B  &= \left\{0,\frac{\beta ^2-1}{\beta  \rho } \right\}  \quad \left[ {} =\left\{0,\frac{2 i}{\rho } \sin  \frac{\pi }{q} \right\} \; \text{if $\beta=e^{i\pi/q}$}\right]
\end{align*}

Typically we lift $A$, $B$, and other matrices to $\SL(2,\IC)$ to make calculations easier. As we are interested only in questions of discreteness this will create no issues for us.

In the case that both $ A $ and $ B $ are primitive parabolics or finite-order elliptics with respective orders $ p $ and $ q $, so $ \tr A = 2\cos \frac{\pi}{p} $ and $ \tr B = 2\cos \frac{\pi}{q} $, we
define $\Gamma^{p,q}_\rho = \langle A,B \rangle $. This gives a natural family of representations $ \IZ_p * \IZ_q \to \PSL(2,\IC) $ parameterised by $ \rho $, namely the homomorphisms defined by sending
the generator of $ \IZ_p $ to $ A $ and the generator of $ \IZ_q $ to $ B $. The interior of the set of $ \rho $ such that the corresponding representations are discrete and free is a quasiconformal
deformation space of orbifold groups. For these representation spaces we introduce the notation
\begin{displaymath}
  \overline{\mathcal{R}}_{p,q} = \{\rho\in\IC:\Gamma^{p,q}_\rho \;\text{is discrete and freely generated by $A$ and $B$}\}.
\end{displaymath}
It is an elementary consequence of J\o rgensen's algebraic convergence theorem \autocite{Jorg} that $\overline{\mathcal{R}}_{p,q}$ is closed. The Riley slice \autocite{KS} has $p=q=\infty$.
Illustrations of two of these spaces are included in \cref{fig1} for $ p = q = 3 $ and $ p = q = 4 $.

\begin{remark}[Production of \cref{fig1}]
  Note that a na\"ive approach based on matrix multiplication would be computationally infeasible.
  The methods described in \autocite{EMS2,EMS3} allow us to overcome these severe computational difficulties. They are based on Farey polynomials, which are trace polynomials associated to Farey words, the latter are obtained via techniques from symbolic dynamics. We refer to these papers for more details. Similar remarks apply to \cref{fig10}.
\end{remark}

\begin{figure}
  \centering
 \includegraphics[width=.8\textwidth]{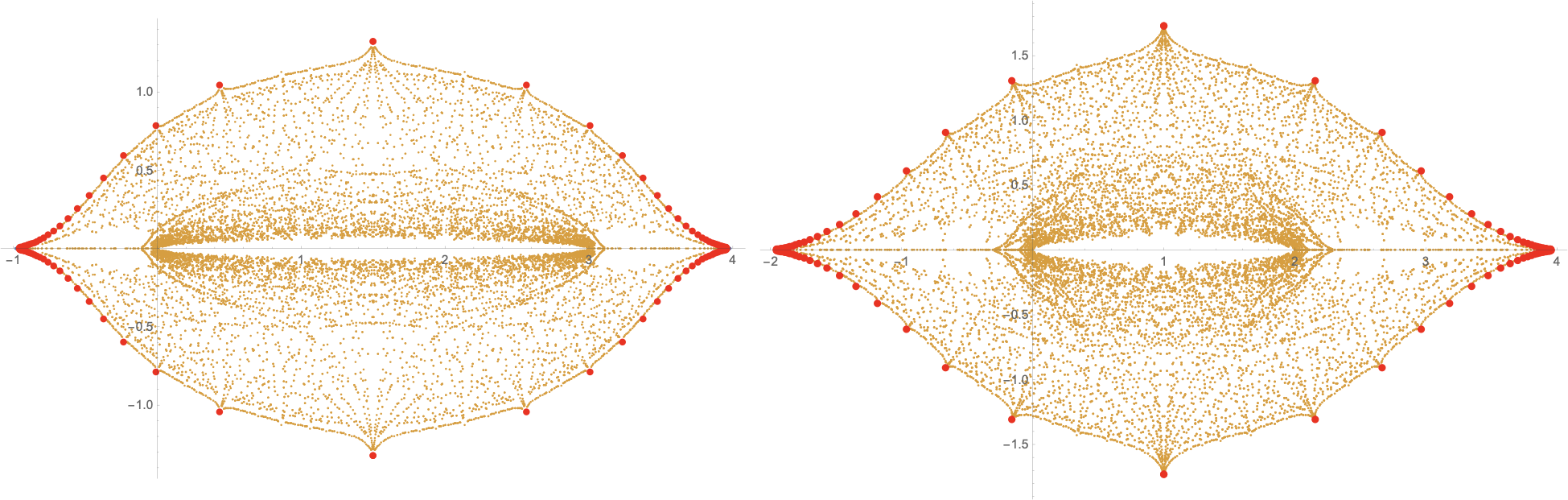}
\caption{The exterior of the illustrated regions is a concrete realisation of the quasiconformal deformation spaces $\overline{\mathcal{R}}_{3,3}$ (left) and $\overline{\mathcal{R}}_{4,4}$. The identified boundary points are certain ``cusp'' groups of slope $\pm \frac{1}{n}$.\label{fig1}}
\end{figure}

We now return to the case that $ A $ and $ B $ are arbitrary matrices of the form in Definition~\ref{XY}. We recall the following well-known parameterisation by traces \autocite[Lemma~2.2 and Remark~2.6]{GMSE}:
\begin{lemma}\label{trace_params}
  The three traces
  \begin{displaymath}
    \tr^2 A - 4,\;\tr^2 B -4,\;\text{and} \; \gamma=\tr\,[A,B]-2
  \end{displaymath}
  determine the group $\langle A,B\rangle$ uniquely up to conjugacy in $\PSL(2,\IC)$ as long as $\gamma\neq 0$. \qed
\end{lemma}
\begin{remark}
  If $\gamma=0$ then $A$ and $B$ have a common fixed point in $\oC$ and therefore $\langle A,B\rangle$ is reducible. Such discrete groups are easily classified \autocite[\S 5.1]{Beardon}.
\end{remark}
If $\langle X,Y\rangle$ is any irreducible subgroup of $\PSL(2,\IC)$, there are $\alpha$, $\beta$, and $\rho$ such that $\langle X,Y\rangle$ is conjugate to $\langle A,B_\rho \rangle$: one can take any values for which
\begin{displaymath}
  \alpha + \alpha^{-1} = \tr X,\; \beta + \beta^{-1} = \tr Y,\;\text{and}\; \rho  \left(  (\alpha-\alpha^{-1})  (\beta-\beta^{-1})+\rho \right) =\tr\, [X,Y]-2.
\end{displaymath}
By Lemma~\ref{trace_params}, the groups $\langle A,B_\rho\rangle $ and $\langle A,B_{\rho'}\rangle$ are conjugate groups if and only if $ \tr\,[A,B_\rho] = \tr\,[A,B_{\rho'}] $.
This allows us to deduce:
\begin{lemma}\label{sym}
  The set $ \overline{\mathcal{R}}_{p,q} $ is symmetric under complex conjugation (i.e.\ $ \langle A, B_{\overline{\rho}} \rangle $ is conjugate to $ \langle A, B_\rho \rangle $
  when $ \alpha = \exp(\pi i/p) $ and $ \beta = \exp(\pi i/q) $). In addition, the function
  \begin{displaymath}
    \sigma(\rho) = -\left(\alpha-\frac{1}{\alpha} \right) \left(\beta-\frac{1}{\beta}\right)-\overline{\rho} = 4 \sin \frac{\pi}{p} \sin \frac{\pi}{q} - \overline{\rho}\,,
  \end{displaymath}
  which is reflection in the vertical line $ \Re z = 2 \sin \frac{\pi}{p} \sin \frac{\pi}{q} $, is a symmetry of $ \overline{\mathcal{R}}_{p,q} $. \qed
\end{lemma}
Notice that the assertion is that the \emph{groups} are conjugate, and not that one can conjugate one generating pair to another (which is usually not the case).

In order to state our main theorem, we set up some notation. For convenience in the following, let $ \xi = \sqrt{2}\sqrt{7-\cos(2\pi/p)} $. We
now define two points in the complex plane:
\begin{align*}
\rho^*_{p,q} &= \left( \cos \frac{\pi }{p} \cos \frac{\pi }{q} +\frac{1}{2} \sin  \frac{\pi }{q}  \left(4 \sin \frac{\pi }{p} +\xi\right) \right)
 + i\left(\frac{\xi}{2}\, \cos \frac{\pi }{q} +\cos \frac{\pi }{p}  \sin \frac{\pi }{q}\right)\\
x_{p,q} &= \frac{\cos  \left[\frac{2\pi}{p}-\frac{2\pi}{q}\right]-\cos \frac{2 \pi }{p}-\cos \frac{2 \pi }{q}+\xi (\sin\frac{\pi }{p}-\sin\left[\frac{\pi}{p}-\frac{2\pi}{q}\right])+5}{ \cos \frac{\pi }{p} \cos\frac{\pi }{q}+\frac{1}{2}\sin \frac{\pi }{q}\left(4 \sin \frac{\pi }{p}+\xi\right)}.
\end{align*}

\begin{definition}\label{defn:omega_bounds}
  Use the points $ \rho^*_{p,q} $ and $ x_{p,q} $ to define a set of twelve lines:
  \begin{enumerate}
    \item the line $\left\{z:\Re z=2+2\cos\left(\frac{\pi}{p}-\frac{\pi}{q}\right)\right\}$ and its image under $ \sigma $ of \cref{sym};
    \item the two lines $\left\{z:\Im z=\pm 2  \sqrt{1-\sin ^2 \frac{\pi }{p}  \sin ^2 \frac{\pi }{q} }\right\}$;
    \item the line $\left\{z: t\rho^*_{p,q}+(1-t)x_{p,q}:t\in \IR\right\}$, its complex conjugate, and the images of both under $ \sigma $;
    \item the line $\left\{z: t\rho^*_{q,p}+(1-t)x_{q,p}:t\in \IR\right\}$, its complex conjugate, and the images of both under $ \sigma $.
  \end{enumerate}
  The set $ \Omega_{p,q} \subset \IC $ is defined to be the interior of the convex polygon bounded by these lines.
\end{definition}

In the special case $p=q$,  there are only $6$ different lines and $x_{p,p}=4$. An example is illustrated in \cref{fig2}, which shows the hexagonal
region that contains the complement of the discrete and faithful representations of $\IZ_4*\IZ_4$.

\begin{figure}
  \centering
  \labellist
  \small\hair 2pt
  \pinlabel {$\rho_{4,4}^*$} [l] at 273 216
  \pinlabel {$x_{4,4}$} [l] at 326 130
  \pinlabel {(1)} [b] at 291 253
  \pinlabel {(2)} [r] at 33 189
  \pinlabel {(3) $=$ (4)} [b] at 196 253
  \endlabellist
  \vspace{2ex}
  \includegraphics[width=0.5\textwidth]{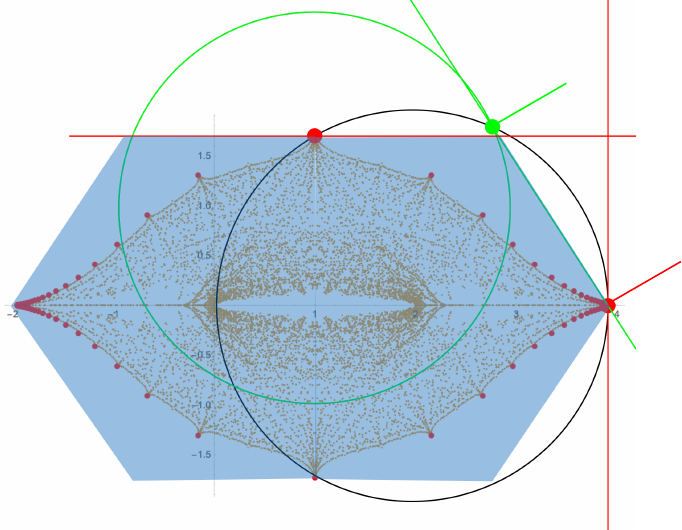}
  \caption{The shaded hexagonal region contains the exterior of the set of faithful discrete representations of $\IZ_4*\IZ_4$. It meets the boundary of the set at the
          cusp groups with slopes $\frac{0}{1} $, $\pm \frac{1}{2} $, and $\frac{1}{1}$. We also label one of each of the lines defined in (1)--(4) of \cref{defn:omega_bounds}; since $ p = q $
          the lines defined in (3) and (4) coincide.\label{fig2}}
 \end{figure}

\section{Conic interactive pairs}\label{sec:pairs}
An \df{interactive pair} $(U,V)$ for a group $G=\langle G_1,G_2 \rangle \subset \PSL(2,\IC)$ as defined by Maskit \autocite[\S VII.A.6]{Mas} is a pair of disjoint subsets $ U $ and $ V $ of $ \hat{\C} $,
both invariant under a subgroup $J<G$ (for us $J$ is trivial), such that every element of $G_1\setminus J$ maps $U$ into $V$ and every element of $G_2\setminus J$ maps $V$ into $U$. In addition, the
pair is \df{proper} if there is a point of $ U $ which is not $ G_2$-equivalent to a point of $ V $ or a point of $ V $ which is not $ G_1$-equivalent to a point of $ U $.

We will study interactive pairs of a very specific form for the groups $ \langle A, B \rangle $ where $ A $ and $ B $ are as in \cref{XY} and where $ A $ is primitive and elliptic of order $ p $.
The case that $A$ is parabolic ($p=\infty$ and $\alpha=1$) follows from taking a limit $ p \to \infty $ in the bounds we achieve for finite $p$,  so we set it aside for the moment.
Any sector of angle $2\pi/p$ with vertex at the fixed point $ \mathbf{P}=(i/2) \csc(\pi/p) \in i\IR_+ $ of $A$ is a fundamental domain for $\langle A \rangle$ acting on $\oC$.

\begin{definition}
  The sector $ \mathbf{K} $, which will be used throughout the paper, is the \emph{open} sector symmetric about the imaginary axis with vertex $\mathbf{P}$, cone angle $\frac{2\pi}{p}$,
  and which contains $0$ in its interior. We say that $\{A,B\}$ has a \df{conic interactive pair} if the pair $ (U, \mathbf{K}) $, where $ U = \hat{\C} \setminus \mathbf{K} $,
  is an interactive pair for $ G_1 = \langle B \rangle $ and $ G_2 = \langle A \rangle $. We note that this is a property of the choice of generators $ \{A,B\} $ and not a property
  of the group $ \langle A,B \rangle $.
\end{definition}

It is immediate that a conic interactive pair $ (U, \mathbf{K}) $ is a proper interactive pair, since $ \mathbf{P} $ is not moved by $ \langle A \rangle $ into $ \mathbf{K} $.
The following lemma is then a special case of \autocite[\S VII.A.10]{Mas}:
\begin{lemma}\label{lem:pingpong}
  Suppose $\{A,B\}$ has a conic interactive pair. Then $\langle A,B\rangle$ is discrete and freely generated by $A$ and $B$,  and $\{A,B^{-1}\}$ has a conic interactive pair. \qed
\end{lemma}
The converse is not true, as shown in \cref{thmx} below.

We now study conjugates of groups with conic interactive pairs which also have conic interactive pairs. Our goal is to generalise work by Lyndon and Ullman, who studied the case of groups
generated by two parabolic elements, say $ \langle X,Y \rangle $; they studied in~\autocite[Theorem 3]{LU} the possible scaling factors by which a fundamental strip for the parabolic $ X $ can be scaled to give a different
fundamental strip for $ X $ and thus a different tessellation of $ \IC $ that is invariant under $ X $, while preserving the property that the complement of the strip is mapped into the strip
by $ B $. In other words, they studied affine transformations which can be used to deform one ping-pong set into another. The extremal affine transformations with this property were then used
to bound the region of discreteness of $ \langle X,Y \rangle $. In our setting we will replace the infinite strip by the cone $\mathbf{K}$ and deform it by scaling to generate a tessellation of a
subset of $ \C $ which is invariant under $A$ and still has complement mapped off itself by a conjugate of $B$; just as in the work of Lyndon and Ullman, the extremal scalings will give us
bounds for the deformation space of $ \langle A,B \rangle $.

Let $\mathbf{P}\in \IC$ and ${\mathcal{P}}$ be a regular $n$-gon whose barycentre is $\mathbf{P}$. Choose an orientation on $\partial{\mathcal{P}}$, which corresponds to a consistent choice of edge for every vertex; from each vertex extend the chosen edge to an infinite ray. We call the result a \df{stellation} of ${\mathcal{P}}$.  This is illustrated in \cref{fig3} for the heptagon and square.

\begin{figure}
\centering
\includegraphics[width=0.5\textwidth]{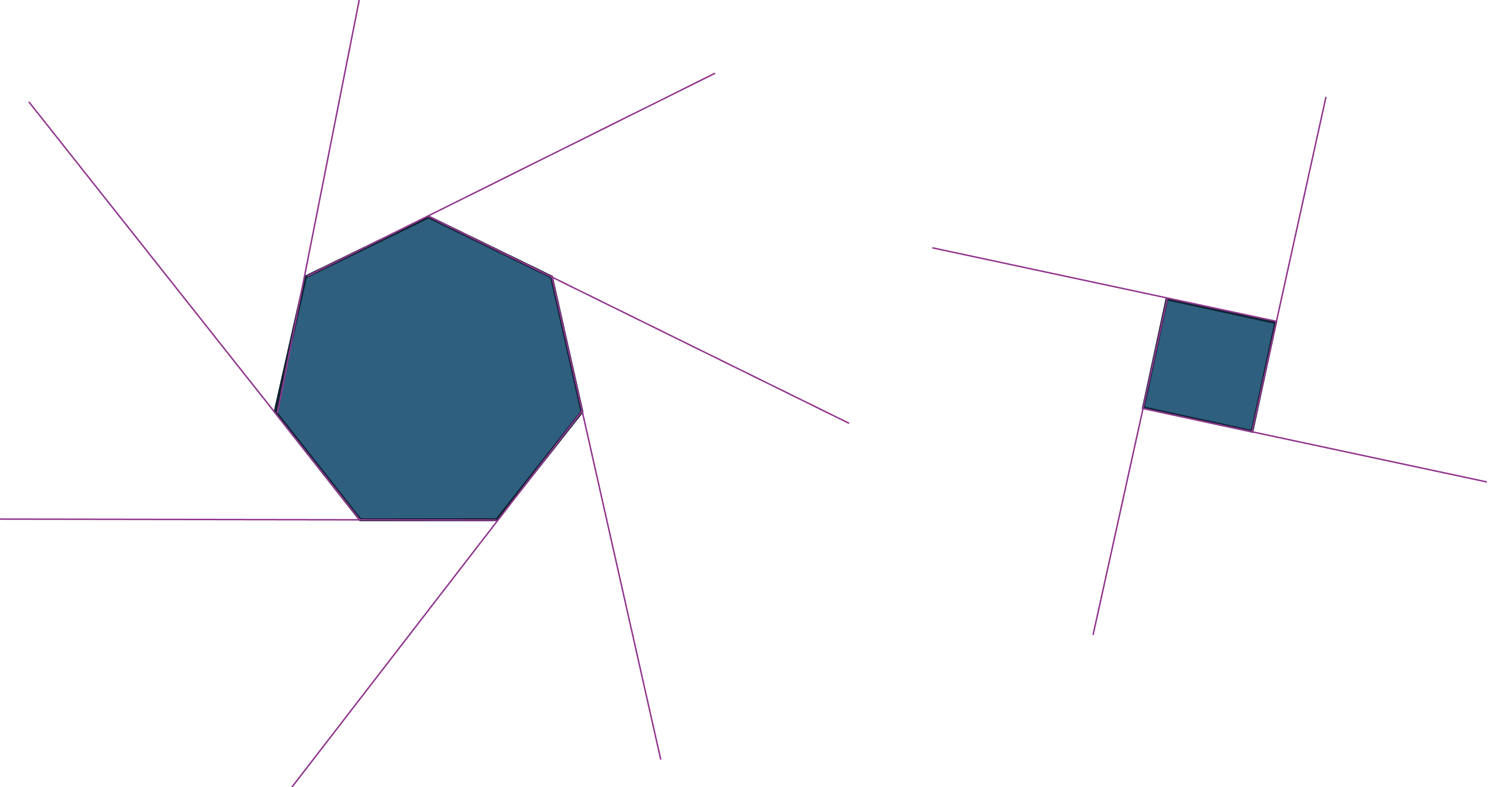}
\caption{Stellations of the regular heptagon the and square.\label{fig3}}
\end{figure}

The components of the complement of the stellation of a regular $n$-gon are sectors of angle $2\pi/n$ and are all isometric. Any rotation of order $n$ fixing the barycentre $\mathbf{P}$ of ${\mathcal{P}}$ and the point $\infty$ permutes these sectors.

\begin{theorem}\label{cip}
  Let $A$ be an elliptic transformation such that $ A $ is of order $ p $ and fixes $ \infty $, and let $ B $ be an arbitrary M\"obius transformation. Suppose that $ \{A,B\} $ admits a
  conic interactive pair with sector $\mathbf{K}$ and consider a stellation of a $p$-gon about the finite fixed point $\mathbf{P}$ of $A$. Let $\Phi$ be a M\"obius transformation such
  that $\Phi(\mathbf{K})$ lies in a component of the complement of the stellation. Then $\langle A,\Phi B \Phi^{-1} \rangle$ is discrete and freely generated by $A$ and $\Phi B \Phi^{-1}$.
\end{theorem}
\begin{proof}
  First observe that $ A $ permutes the components of the complement, and therefore $ A^n(\Phi(\mathbf{K})) \cap \Phi(\mathbf{K}) = \emptyset $ whenever $ A^n \neq \mathrm{Id} $. We need also to check
  \begin{displaymath}
    \Phi B^n \Phi^{-1} (\mathbb{C} \setminus \Phi(\mathbf{K})) = \Phi B^n (\mathbb{C} \setminus \mathbf{K}) \subseteq \Phi(\mathbf{K})
  \end{displaymath}
  whenever $ B^n \neq \mathrm{Id} $, and hence $ (\Phi(\mathbf{K}), \mathbb{C} \setminus \Phi(\mathbf{K})) $ gives an interactive pair (which is proper since $ \Phi(\mathbf{P}) $ is not moved
  into $\Phi(\mathbf{K})$ by the action of $ \Phi B^n \Phi^{-1} $).
\end{proof}
\begin{remark}
  In the case that $ n = 2 $ a stellation is just the complement of a line.
\end{remark}

We can apply this result to improve estimates on the discreteness locus. Suppose that we know that $ \langle A, B \rangle $ is discrete. We then make a judicious choice of $\Phi$, obtain
a new group $ \langle A, \Phi B \Phi^{-1} \rangle $ which is guaranteed to be discrete, and then compute the new value of $ \gamma $ (which is the conjugacy invariant $\tr\,[X,\Phi Y \Phi^{-1}]-2$
given in Lemma~\ref{trace_params}). If all goes well, this new value of $ \gamma $ is outside our old bounds of discreteness, and so we can enlarge those bounds.

Following Lyndon and Ullman, our first choice for $\Phi$ is multiplication by a constant and the sector is the sector $\mathbf{K}$ with apex $ \mathbf{P} $. This allows us to prove the following
corollary.

\begin{corollary} Let $ A $ be a primitive elliptic transformation with order $ p > 2 $, let $ Y \in \PSL(2,\IC) $  be  arbitrary apart from not sharing a fixed point with $A$,
and suppose that $ \{A,Y\} $ admits a conic interactive pair with sector $\mathbf{K}$. Let $S : z\mapsto \mu z$ where $\mu=1+it$ for some $t\in \IR$.
Then $\langle A,S^{-1}YS \rangle$ is discrete and freely generated by $A$ and $S^{-1} Y S$.
\end{corollary}
\begin{proof}
For $\mu\neq 0$ let $S$ denote the loxodromic map $S:z\mapsto \mu z$. Let $A$ be primitive elliptic of order $p$ as above at  \cref{XY} with sector $\mathbf{K}$ and apex $\mathbf{P}$.  Write $\mu=\lambda e^{i\theta}$ where $\lambda \in \IR^* $ and $\theta\in \IR$. We want to calculate values for $\lambda\in (0,1]$ and $\theta$ that imply  that the translates of  $(1/\mu) K$ (the red sectors identified in \cref{fig4}) share edges. This then gives the stellation we seek.

\begin{figure}
\centering
\includegraphics[width=0.6\textwidth]{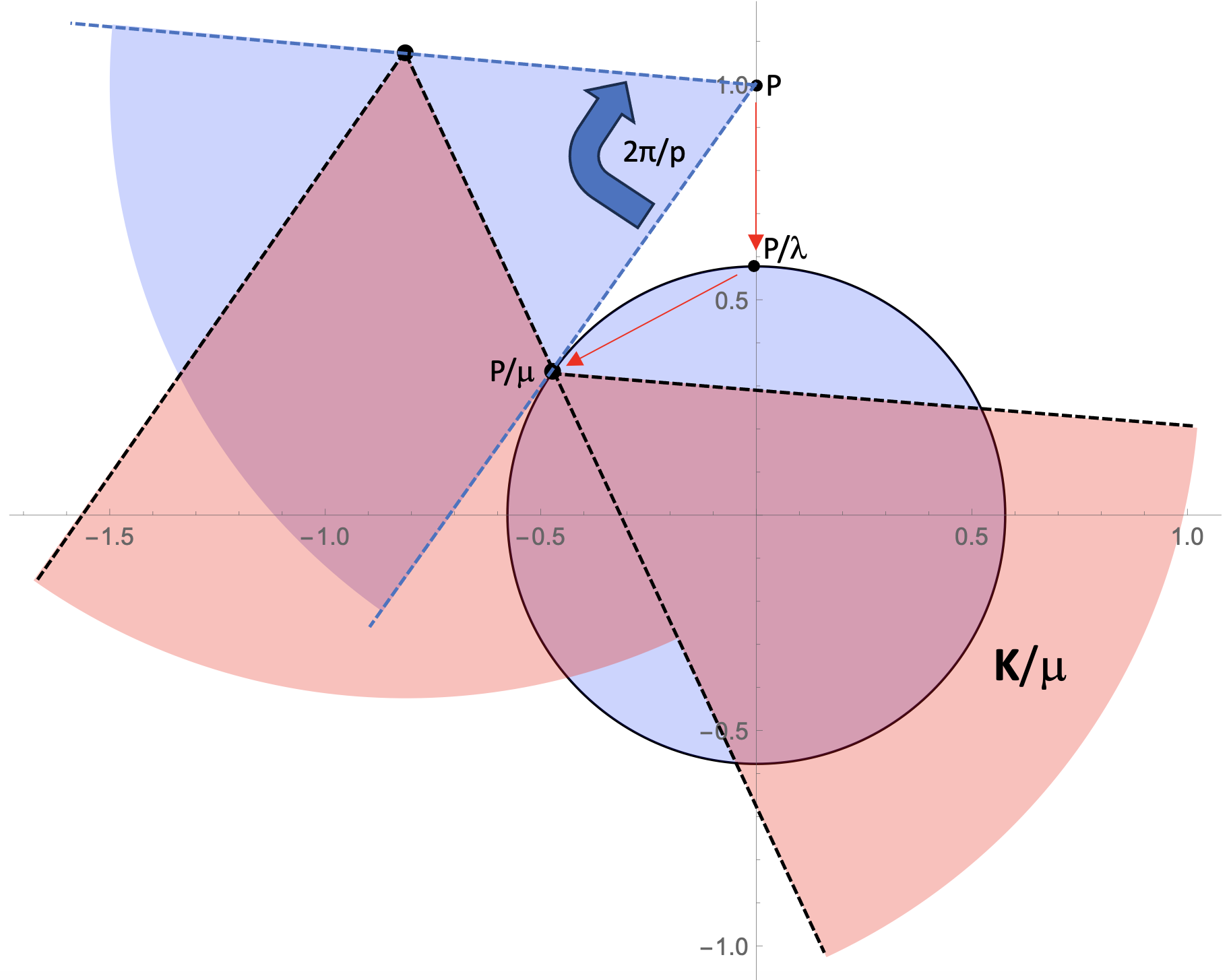}
\caption{The translates of the sector $\mathbf{K}/\mu$ line up and are disjoint. The sector with apex $\mathbf{K}/\mu$ maps onto the sector in the top left under rotation by $A$.\label{fig4}}
\end{figure}

Consideration of \cref{fig4}, along with the periodicity of $A$,  shows that a necessary and sufficient condition for these edges to coincide is that the triangle with vertices $\mathbf{P}$, $\mathbf{P}/\mu$, and $0$ has a right angle at $\mathbf{P}/\mu$.  That is, $(0-\mathbf{P}/\mu)\cdot\overline{ (\mathbf{P}-\mathbf{P}/\mu)}$ is imaginary. Then $1-\mu$ is imaginary, i.e.\ $ 1 - \mu \in \i \R $ and so
the desired condition can be written as $ \mu = 1 + i t $ for $ t \in \IR $. The result now follows from combining this calculation with \cref{cip}.
\end{proof}

\begin{theorem}\label{thm4} Let $ A $ be a primitive elliptic transformation with order $ p>2 $, let $ Y \in \PSL(2,\IC) $  be  arbitrary apart from not sharing a fixed point with $A$,
and suppose that $ \{A,Y\} $ admits a conic interactive pair with sector $\mathbf{K}$. Let $B_{\rho_0}$ be as defined at \cref{XY} with $\tr(B_{\rho_0})=\tr(Y)$ and $c =Y_{2,1}$.
For $t\in \IR$,  let $ \rho_t = \rho_0 + i c t $. Then $\langle A,B_{\rho_t} \rangle$ is discrete and freely generated by $A$ and $B_{\rho_t}$.
\end{theorem}
\begin{proof}
Let
\begin{displaymath}
  S = \begin{pmatrix} \sqrt{1+it} & 0 \\ 0 & 1/\sqrt{1+it} \end{pmatrix}.
\end{displaymath}
The group $\langle A, S^{-1}YS\rangle$ is conjugate to a group of the form $\langle A,B_{\rho_t}\rangle$ for some $t$ by the conditions provided in \cref{tracechecks}, and the former is discrete by \cref{cip} and our choice of $S$. Since $\tr\,[A, B_{\rho_0}]=\tr\,[A, Y]$, $\rho_t = \rho_0 + i c t$ implies $\tr\,[A, B_{\rho_t}] =  \tr\,[A, S^{-1}YS]$ and the proof is complete.
\end{proof}

Geometrically,  the set $\{\rho_t:t\in \IR\}$ is a line passing through $\rho_0$ which is orthogonal to the line through $0$ and $c$, since $( ict) \overline c  =   it\abs{c}^2$ is purely imaginary.

If $c=\rho_0$,  we simplify \cref{thm4} to obtain the following.
\begin{corollary}\label{abconic} Let $\{A,B_\rho\}$ (defined as at \cref{XY}, with $ p > 2 $) admit a conic interactive pair with sector $\mathbf{K}$. Then for all $t$,  $\langle A,B_{\rho (1+it)}\rangle$ is discrete and freely generated by $A$ and $B_{\rho(1+it)}$. \qed
\end{corollary}

\begin{example}\label{thmx}
  Consider the group $ \Gamma^{p,q}_{\rho} $, where
  \begin{displaymath}
    \rho = \rho_{1/2} =2 \left( \sin \frac{\pi }{p}  \sin  \frac{\pi }{q} \pm i \sqrt{\sin ^2 \frac{\pi }{p}  \sin ^2 \frac{\pi }{q} -1}\right).
  \end{displaymath}
  This group is the cusp group on the boundary of the deformation space corresponding to pinching the curve of slope $ 1/2 $ on the $4$-marked sphere;
  since the word representing this curve is the commutator $ [A,B_\rho] $ (see~\cite{EMS2}) it is obtained by choosing the correct solution to the equation $ \tr\,[A,B_\rho] = -2 $.
  As it is a cusp group, it is discrete with a circle packing limit set, and the components of its domain of discontinuity are stabilised by conjugates of $(p,q,\infty)$-triangle
  groups. One can see discreteness directly by drawing the isometric circles of $ A $ and $ B_\rho $ and doing cutting and pasting to transform them into a fundamental domain.

  The line through  $\rho_{1/2}$ in the direction of $i\rho_{1/2}$ meets the real line at $ \rho_0 = 4\sin(\pi/p) \sin(\pi/q) $ and, while $\Gamma_{\rho_{1/2}}$ is a discrete group,
  $ \Gamma_{\rho_0} $ is not (this follows from the classification of Fuchsian groups generated by two parabolics \autocite{Knapp}). Applying the contrapositive of \cref{abconic}, the
  group $\Gamma_{\rho_{1/2}}$ is discrete and conjugate to $\langle A,B_{\rho_{1/2}}\rangle$, but has no interactive conic pair coming from that generating set.
\end{example}

\section{Bounds using isometric circles}\label{sec:families}

We fix $ A $ as at \cref{XY} with $ p > 2 $ and we fix an element $ Y\in \PSL(2,\IC) $ which does not share any fixed points with $ A $.
We will identify sets of matrices $Y$ so that $\{A,Y\}$ has a conic interactive pair. In the classical theory of Kleinian groups, the first guess for an
interactive pair (or a ping-pong set in general) tends to involve the isometric discs---this goes back at least as far as Brenner~\cite{Brenner} in the current
setting, and to the classic works of Fricke and Klein in general---and in this case the most na\"ive condition one can consider is that the
isometric discs of $Y$ lie in $\mathbf{K}$. We first establish bounds using isometric disks in \cref{lem:first_bound}. We then focus on the case of two elliptic generators.
In \cref{lem2} we find necessary and sufficient conditions for the translates of $\IC\setminus \mathbf{K}$ under $B$ to be pairwise disjoint. If these conditions are fulfilled we find an interactive pair for $\{A,B\}$ with sector $\mathbf{K}$. Carefully expressing these conditions leads to estimates on $\rho$ in \cref{Vpq}.

The \df{isometric discs}  of
\begin{displaymath}
  Y = \begin{bmatrix} a&b \\ c&d \end{bmatrix}
\end{displaymath}
are the two discs
\begin{displaymath} D_1=\{z:\abs{c z+d}\leq 1\}, \quad D_2= \{z:\abs{c z-a} \leq1\},\end{displaymath}
with centres $Q_1=-d/c$ and $Q_2=a/c$ and common radius $1/\abs{c}$. These discs are important for us since they are convenient to define in terms of the matrix
entries and they have the property that
\begin{equation}\label{ICprop}
\begin{array}{ll}
  Y(D_1)=\oC\setminus \operatorname{int}(D_2),& Y^{-1}(D_2)=\oC\setminus \operatorname{int}(D_1),\\
  Y(Q_1)=\infty,& Y^{-1}(Q_2)=\infty,
\end{array}
\end{equation}
so that $Y(\oC\setminus(D_1\cup D_2)) \subset D_1\cup D_2$. For proofs of these properties of isometric circles see~\cite[\S I.C]{Mas}.

\begin{lemma}\label{meetK} There exist integers $m$ and $n$ so that  $\widetilde{Y}=A^nYA^m$ has isometric discs $D_1$ and $D_2$ whose centres both lie in $\mathbf{K}$ and whose radius is the same as that of the isometric discs of $Y$.
 \end{lemma}
\begin{proof}
  As $\mathbf{K}$ is a fundamental domain for $A$,  there are $m$ and $n$ so that $A^{-m}(D_1)$ and $A^n(D_2)$ have their respective centres in $\mathbf{K}$.
  Set $\widetilde{D}_1=A^{-m}(D_1)$, $\widetilde{D}_2=A^{n}(D_2)$. Then
  \begin{displaymath}
    A^{n}BA^{m}(\widetilde{D}_1)= A^{n}BA^{m}(A^{-m}(D_1))=A^{n}(\oC\setminus D_2)=\oC\setminus \widetilde{D}_2.
  \end{displaymath}
  Further $B $ acts isometrically on $\partial D_i$, and $A$ is a Euclidean isometry,  so these discs have the same radius and so must be the isometric discs.
\end{proof}
The utility of \cref{meetK} is that for all $n$ and $ m $, $\tr\,[A,A^nYA^m]=\tr\,[A,Y]$ and therefore $\langle A, Y\rangle =\langle A,\widetilde{Y}\rangle$.

\subsection{First bounds via isometric discs}
In the special case that $Y$ is elliptic,  the fixed point $0$ lies in both $D_1 $ and $ D_2 $ so both discs meet $\mathbf{K}$. \Cref{meetK} allows us to modify our generators in order to arrange
this desirable configuration more generally, and therefore the following result (which generalises a result of Chang, Jennings, and Ree \autocite{CJR} who proved a similar theorem for the\
case of two parabolic generators) gives us a bound for arbitrary $ Y $.

\begin{lemma}\label{lem:first_bound}
Let $ A $ and $ Y $ be given as at the start of the section and suppose that the isometric discs of $Y$ meet $\mathbf{K}$. If
\begin{itemize}
\item $A(D_1)\cap D_1 = \emptyset$, that is if and only if $\abs{c-d (\alpha-\overline{\alpha})}>2$, and
\item $A(D_2)\cap D_2 = \emptyset$, that is if and only if $\abs{c+a (\alpha-\overline{\alpha})}> 2$, and
\item $A(D_1)\cap D_2 = \emptyset$, that is if and only if $\abs{c+a\alpha+d\overline{\alpha}}> 2$, and
\item $A(D_2)\cap D_1 = \emptyset$, that is if and only if $\abs{c-a \overline{\alpha} - d \alpha}> 2$,
\end{itemize}
then $\langle A,Y\rangle$ is discrete and is the free product of $ \langle A \rangle $ and $ \langle Y \rangle $.
\end{lemma}
(The proof of the lemma also goes through when $ p = 2 $.)
\begin{proof}
  The equivalences between non-intersection of the discs and the inequalities are easy to prove. For instance, observe that $ D_1 $ has centre $ -d/c $
  and radius $ 1/\abs{c} $; the matrix $ A $ acts as a Euclidean rotation and so the image $ A(D_1) $ has the same radius as $ D_1 $ and centre $ A(-d/c) = \alpha^2(-d/c) + \alpha $.
  Thus $ A(D_1) \cap D_1 = \emptyset $ if and only if $ \abs{ (-d/c) - (\alpha^2(-d/c) + \alpha) } > 1/\abs{c} + 1/\abs{c} $; clearing denominators, this is equivalent to
  \begin{displaymath}
    \abs{-d + \alpha^2 d - \alpha c} > 2
  \end{displaymath}
  and multiplying through by $ 1 = \abs{\alpha^{-1}} = \abs{\overline{\alpha}} $ this becomes
  \begin{displaymath}
    \abs{ -d\overline{\alpha} + \alpha d - c } > 2
  \end{displaymath}
  which is clearly equivalent to the first inequality in the statement of the lemma. The other three equivalences are similar.

  To prove the lemma proper, we perform a surgery on the fundamental domain $ \mathbf{K} $ for $ \langle A \rangle $ by adding the fundamental discs of $ Y $ and deleting their $ A$-translates:
  \begin{displaymath}
    V =\left(\mathbf{K} \cup D_1 \cup D_2\right) \setminus (A^{-1}(D_1)\cup A(D_2)),
  \end{displaymath}
  where we may have to exchange one or both of $A$ for $A^{-1}$ depending on the signs of $\Re (-d/c)$ and $\Re (a/c)$. This situation is illustrated in \cref{fig5}.

  \begin{figure}
  \centering
  \includegraphics[width=\textwidth]{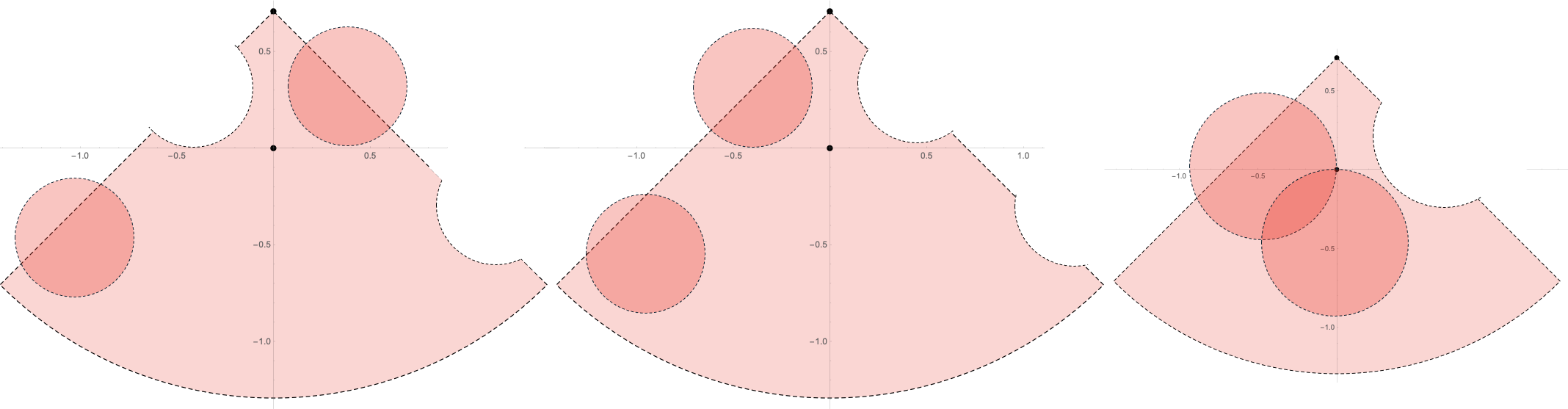}
  \caption{The new fundamental domain for $A$ consists of $\mathbf{K}$ together with the isometric discs of $Y$ and suitable images deleted. Left: The real part of the centres of discs have opposite signs.  Middle: Real part of centres have same (negative) sign. Right: Elliptic example.\label{fig5}}
  \end{figure}

  Let $ U = \hat{\mathbb{C}} \setminus V $. By the properties of isometric circles enumerated in \eqref{ICprop}, $ Y $ sends the exterior of $ D_1 $ into the interior of $ D_2 $. By
  construction, $ U $ lies in the exterior of $ D_1 $, and so $ U $ is moved off itself by $ Y $. Similarly, $ U $ is moved off itself by $ Y^{-1} $. By our assumptions on $A,D_1$ and $D_2$ we therefore
  have an interactive pair $ (U,V) $ and the proof is complete via \autocite[\S VII.A.10]{Mas} if we can show that it is proper. Let $ p $ be a fixed point of $ Y $; then $ p $ is
  not $ \langle Y \rangle$-equivalent to any point of $ U $, since the fixed points of $ Y^n $ all lie in the isometric discs of $ Y $.
\end{proof}

If $0$ is a fixed point for $Y$ (that is if $b=0$), then the four inequalities in \cref{lem:first_bound} represent discs about $c$ of radius $2$ as $a$ and $d$ are determined by the trace ($\beta+1/\beta$) and determinant ($ad=1$). A drawback of this construction is that to apply our earlier results on conic interactive pairs we really do need $\mathbf{K}$ unaltered and so must modify the isometric discs of $Y$, unless they already lie in $\mathbf{K}$.  In doing so the equations become so complicated that they are of little utility. However there is a special case where we can do this in a straightforward manner, and fortunately it is the case of primary interest to us.

\subsection{The case that $A$ and $B$ are elliptic}
Let $A$ and $B=B_\rho$ be defined as at \cref{XY} with $ \alpha = e^{i\pi/p} $ and $\beta=e^{i\pi/q}$ where $ p > 2 $. Let $ J $ be the involution $ J(z) = 1/z $.

\begin{lemma}\label{lem2} Necessary and sufficient conditions for the translates of $\IC\setminus \mathbf{K}$ under $B$ to be pairwise disjoint are the following:
\begin{displaymath}
\begin{array}{ll}
 \abs{\alpha  \left(\beta -\overline{\beta }\right)+\rho}>2, &
 \abs{\beta  \overline{\alpha }+\alpha  \overline{\beta }-\rho }>2, \\
 \abs{\overline{\alpha } \left(\overline{\beta }-\beta \right)+\rho}>2, &
 \abs{\overline{\alpha } \overline{\beta }+\alpha \beta +\rho}>2.
\end{array}
\end{displaymath}
If these four conditions hold for $\alpha$, $\beta$ and $\rho$,  then $\{A,B\}$ admits a conic interactive pair with sector $\mathbf{K}$.
\end{lemma}
\begin{proof}
It is clearly equivalent to identify conditions for when the translates of $J(\oC\setminus \mathbf{K})=\oC\setminus J(\mathbf{K})$ under $\widetilde{B}=JBJ$ are pairwise disjoint.
First observe that
\begin{displaymath}
  \widetilde{B} = \begin{pmatrix} \overline{\beta} & \rho \\  0 & \beta  \end{pmatrix}\; \text{which represents}\; \widetilde{B} : z\mapsto \overline {\beta}^2 z+\overline {\beta}\rho
\end{displaymath}
is a Euclidean isometry, a rotation with fixed point $\widetilde{P} = - (i\rho/2) \csc(\pi/q)$.

Next we want to identify $J(\mathbf{K})$.  The boundary of this region must be symmetric across the imaginary axis, and must be the intersection of two discs which meet at angle $2\pi/p$ at the points $0=J(\infty)$ and $J(P)=1/P = -2i\sin (\pi/p)$.  The centre of these discs must lie on the line $\Im(z)=-\sin (\pi/p)$ and both discs must have the same radius.  This means they are isometric discs for a rotation of order $p$ with fixed points $0$ and $-2i\sin (\pi/p)$. That must be
\begin{displaymath}
\widetilde{A}=JAJ= \begin{pmatrix} \overline{\alpha} & 0 \\ 1 & \alpha \end{pmatrix}\;\text{which represents}\; \widetilde{A}:z\mapsto \frac{\overline{\alpha} z}{z+\alpha}.
\end{displaymath}
The transformation $\widetilde{A}$ has isometric discs $\widetilde{D_1}=\{z:\abs{z+\alpha}<1\} $ and $ \widetilde{D_2}=\{z:\abs{z-\overline{\alpha}}<=>1\}$ and fixed points $0$ and $(i/2)\csc(\pi/p)$. Then $J(\mathbf{K})=\oC\setminus (\widetilde{D}_1\cup \widetilde{D}_2)$.
We can now make exactly the same calculation as we did previously,  considering when these discs are moved off themselves.
This gives us the four inequalities from the lemma statement, and gives us an interactive pair with cone $ \mathbf{K} $ (moved into $ \mathbb{C} \setminus \mathbf{K} $ by $ A $) and $ U = \mathbb{C} \setminus \mathbf{K} $
(moved into $ \mathbf{K} $ by $ B $).
\end{proof}

From the information in the proof, we may construct systems of discs similar to those which appeared in the proof of \cref{lem:first_bound}. As in that proof,
we consider a surgery on a sector. Let $\widetilde{K}$ be the sector symmetric about the imaginary axis with apex $\widetilde{P} = - (i\rho/2) \csc(\pi/q)$ that
contains the origin (i.e.\ a fundamental domain for $ \widetilde{B} $), and add and delete translates of $ \widetilde{D_i} $ to obtain a set
\begin{displaymath}
  \widetilde{V} = \left(\widetilde{K} \cup \widetilde{D_1}\cup \widetilde{D_2}\right)\setminus \left(\widetilde{A}(\widetilde{D_1})\cup\widetilde{A}^{-1}(\widetilde{D_2})\right)
\end{displaymath}
(where possibly we might need to replace $ A $ with $ A^{-1} $ or vice versa, depending on the circle arrangements). Then, just as in the proof of \cref{lem:first_bound},
$ \widetilde{V} $ and its complement form a proper interactive pair.

We observe that, when $ p = q $, the bounds are the same as those which were constructed in \cref{lem:first_bound}. Illustrated in \cref{discfamilies} are a few of the families of discs which arise as bounds.
With $\alpha=e^{i\pi/p}$ and $\beta=e^{i\pi/q}$ we write the conditions as
\begin{displaymath}
  \begin{array}{ll}
    \abs{2i e^{i\pi/p}  \sin\frac{\pi}{q}+\rho}>2, &
    \abs{2\cos\left( \frac{\pi}{p}-\frac{\pi}{q}\right)-\rho}>2, \\
    \abs{2i e^{-i\pi/p} \sin\frac{\pi}{q}-\rho}>2, &
    \abs{2\cos\left( \frac{\pi}{p}+\frac{\pi}{q}\right)+\rho}>2.
  \end{array}
\end{displaymath}
There is quite a strong lack of symmetry between $\alpha$ and $\beta$ in these discs.  This is evidenced by \cref{fig8}, where we show the discs which arise for the deformation spaces $ \mathcal{R}_{7,3} $
and $ \mathcal{R}_{3,7} $, which are equal to each other as subsets of $ \C $.  We will use this asymmetry below to improve certain estimates.

\begin{figure}
  \centering
  \includegraphics[width=\textwidth]{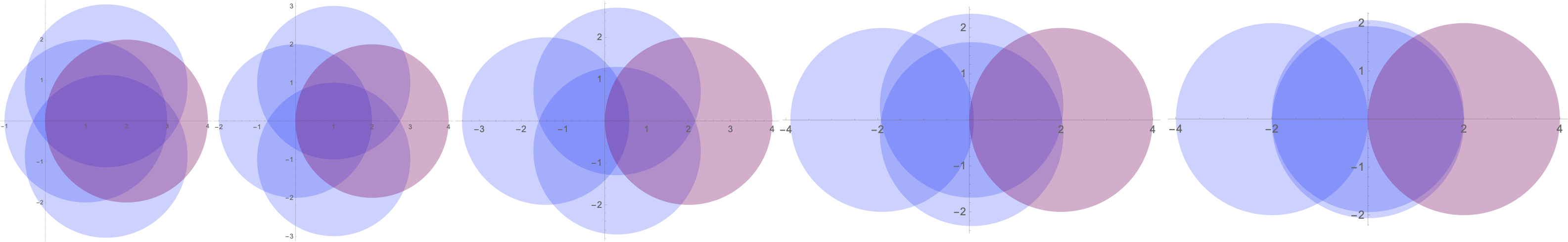}
  \caption{Families of discs from \cref{lem2}, $p=q$. From left: $p=3,4,8,20,100$.  A disc of radius $2$ is always included when $p=q$ and is shaded red. \label{discfamilies}}
\end{figure}

\begin{figure}
  \centering
  \includegraphics[width=0.5\textwidth]{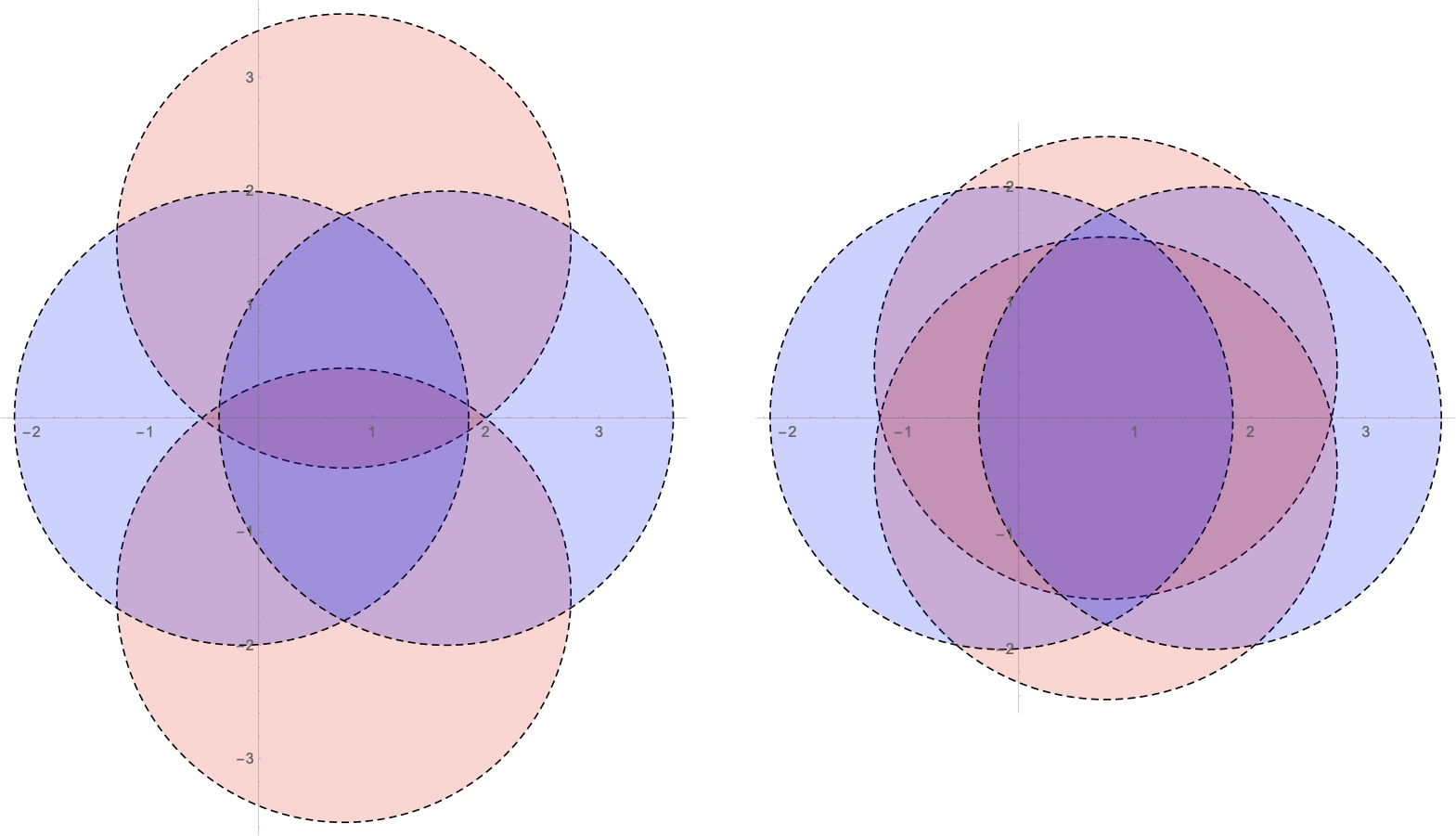}
  \caption{The bounds are asymmetric in $ p $ and $ q $. Left: $p=7$ and $q=3$. Right: $p=3$ and $q=7$.\label{fig8}}
\end{figure}

The following lemma, which explains the definitions of $ \rho^*_{p,q} $ and $ x_{p,q} $ in \cref{defn:omega_bounds}, is shown by direct computation.
\begin{lemma}
  The point $\rho^*_{p,q}$  lies on the intersection of the two circles
  \begin{displaymath}
    \abs{2i e^{i\pi/p}  \sin\frac{\pi}{q}+\rho}=2 \;\text{and}\; \abs{2\cos\left( \frac{\pi}{p}-\frac{\pi}{q}\right)-\rho} =2
  \end{displaymath}
  and lies outside the other two discs of \cref{lem2}. The line  $\{\rho+it\rho:t\in\IR\}$ meets the real axis at the point $x_{p,q}$. \qed
\end{lemma}

Notice that $\rho^*_{p,q}\neq \rho^*_{q,p}$ in general.   At this point we can give first estimates on the values of $\rho$, which is most of our main result \cref{thm:main}.

\begin{theorem}\label{Vpq}
  Let $X,Y\in \SL(2,\IC)$ be primitive elliptic elements of order $p$ and $q$, $3\leq p\leq q$. Let $\rho$ be a solution of
  \begin{displaymath}
    \rho  \left( \rho-4\sin \frac{\pi}{p}\sin\frac{\pi}{q} \right) = \tr\,[X,Y]-2.
  \end{displaymath}
  Let $ \Omega'_{p,q} $ be the convex polygon determined by the lines in (1), (3), and (4) of \cref{defn:omega_bounds} together with
  \begin{enumerate}
    \item[(2$'$)] the two lines $\left\{z:\Im z=\pm 2 \pm 2 \cos \frac{\pi}{q} \sin \frac{\pi}{q} \right\}$.
  \end{enumerate}
  If $\rho\not\in \Omega'_{p,q}$, then $\langle X,Y\rangle$ is discrete and free on the given generating set.
\end{theorem}
\begin{proof}
  The assumption on $\rho$ implies that $\langle X,Y \rangle$ is conjugate to $\langle A,B_\rho \rangle$ with $A$ and $B=B_\rho$ as at \cref{XY}. The bounds on the real part and imaginary parts of $\rho$,
  i.e.\ condition (1) in \cref{defn:omega_bounds} and condition (2$'$) in the theorem statement, follow since the four excluded discs lie in the intersection of these strips. Next, the
  points $\rho^*_{p,q}$ and $\rho^*_{q,p}$ lie in the positive quadrant of $\IC$ and on the boundaries of the union of these four discs (though two of the discs are different and come from
  interchanging $p$ and $q$).  \Cref{abconic} now applies as our calculations show the remaining  lines pass through $\rho^*_{p,q}$ in the direction of $i\rho^*_{p,q}$. This gives us
  conditions (3) and (4) of \cref{defn:omega_bounds}.
\end{proof}

\Cref{fig9} indicates the lines and discs for a range of $p$ and $q$. It is clear that some improvements can be made.  For instance \cref{abconic} indicates the entire envelope of lines
through the points  $\gamma=2+2 \cos\left(\frac{\pi}{p}-\frac{\pi}{q}\right)+2\cos(\theta)+i \sin(\theta)$ in the direction $i\gamma$ ($0\leq \theta<\arg(\rho^*_{p,q})$), i.e.\ points on
the boundary of the  right-most disc with real part larger than $\rho^*_{p,q}$, lie in the free space. In practice this is seldom much better than what we have illustrated when one of $p$ or $q$
is large. Also,  the real part bound is sharp and holds with equality for the $(p,q,\infty)$ triangle group and its image under $ \sigma $ of \cref{sym}, which are the points
\begin{displaymath}
  \rho = \pm2 \pm 2\cos\left(\frac{\pi}{a} \mp \frac{\pi}{b}\right)\!.
\end{displaymath}
The imaginary part bound is asymptotically sharp as $p,q\to \infty$, and converges to the slope $1/2$ cusp group of \cref{thmx} on the boundary of the $ \mathcal{R}_{p,q} $. Different arguments
are needed to give a horizontal strip with sharp bounds at this cusp, and this is the content of the next section.

\begin{figure}
\centering
\includegraphics[width=\textwidth]{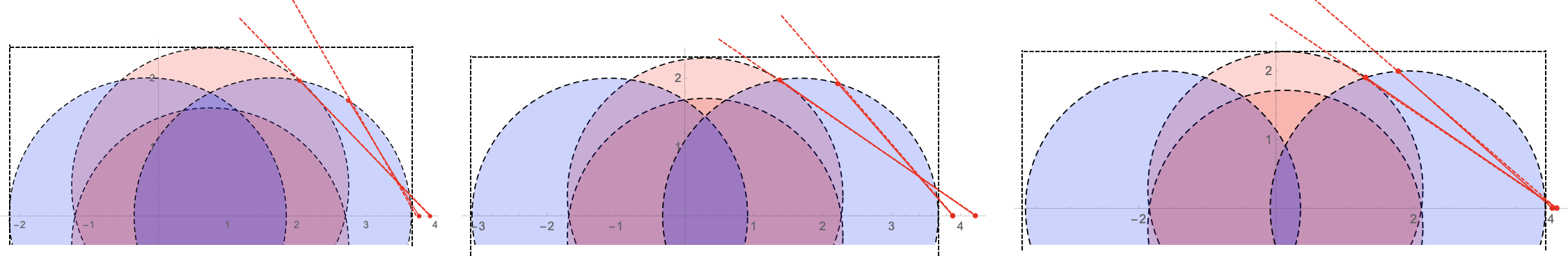}
\caption{The region $\Omega_{p,q}$ intersecting the positive quadrant for a range of $p$ and $q$. From left: $(p,q)=(3,7),\; (4,15), \; (7,20)$.\label{fig9}}
\end{figure}

\section{A different parameterisation of the character variety}\label{sec:ref}
Isometric circles are not conjugacy invariants: they measure the action of the transformations with respect to Euclidean geometry rather than conformal geometry, so
if $ X $ is a M\"obius transformation with isometric circles $ C $ and $ D $, and $ Y $ is any M\"obius transformation which does not fix $ \infty $, then the isometric
circles of $ YXY^{-1} $ will in general \emph{not} be $ Y(C) $ and $ Y(D) $. Thus choosing a different conjugacy representative for our groups changes the bounds obtained. Set
\begin{displaymath}
  U = \begin{pmatrix}
  \cos \frac{\pi }{p}  & -\sin\frac{\pi }{p}  \\
  \sin \frac{\pi }{p}  & \cos \frac{\pi }{p}  \\
  \end{pmatrix}\;\text{and}\;
  V=\begin{pmatrix}
  \cos \frac{\pi }{q}  & -\lambda  \sin  \frac{\pi }{q}   \\
  \lambda^{-1} \sin  \frac{\pi }{q} & \cos  \frac{\pi }{q}  \\
 \end{pmatrix}\!.
\end{displaymath}

Interchanging $\lambda$ with $-\lambda$ changes $V$ to $V^{-1}$ and preserves discreteness. Since
\begin{displaymath}
  \tr\, [U,V_\lambda]-2 = \left(\lambda-\frac{1}{\lambda}\right)^2\sin^2 \frac{\pi }{p}  \sin^2  \frac{\pi }{q},
\end{displaymath}
interchanging $\lambda$ and $1/\lambda$ does not change the conjugacy class and the formula is symmetric in $p$ and $q$, so we may always assume that $ \abs{\lambda} \geq 1 $ and $ p \leq q $.
These conditions are used in the choices made later in the section.

We can derive the relation between $ \lambda $ and $ \rho $ when the two groups $\langle X,Y_\rho\rangle$ and $\langle U,V_\lambda\rangle$ are conjugate. We
use the conjugacy invariance of the commutator trace to obtain the equation
\begin{displaymath}
  \tr\, [X,Y_\rho]-2 = \rho  \left(\rho -4 \sin \frac{\pi }{p} \sin \frac{\pi }{q} \right) =(\lambda-1/\lambda)^2\sin^2 \frac{\pi }{p}\sin^2 \frac{\pi }{q}
\end{displaymath}
whose two solutions are given by
\begin{equation}\label{rholambda}
  \rho = -\frac{(\lambda -1)^2}{\lambda } \sin \frac{\pi }{p}  \sin \frac{\pi }{q}\;\text{and}\; \rho = \frac{(\lambda +1)^2}{\lambda } \sin  \frac{\pi }{p}  \sin  \frac{\pi }{q}
\end{equation}

The fixed points of $U$ are $\pm i$, the centres of its isometric circles are $\pm\cot \frac{\pi }{p}$, and each isometric disc has radius $\csc\frac{\pi }{p}$.  For $V$ the fixed points are $\pm \lambda$ and
the isometric disc have centres $\pm\lambda \cot \frac{\pi }{q}$ with radii $\abs{\lambda}\csc\frac{\pi }{q}$. A sufficient condition for the group to lie (up to conjugacy) in $\mathcal{R}_{p,q}$ is that the
isometric discs of $U$ both lie in the intersection of the isometric discs of $V$. This gives us:
\begin{lemma}\label{ineq} Suppose that $\lambda$ lies inside the region defined by the inequality
  \begin{displaymath}
    \abs{\lambda\cot \frac{\pi }{q} \pm  \cot \frac{\pi }{p}}+ \csc\frac{\pi }{p} \leq \abs{\lambda} \csc\frac{\pi }{q}.
  \end{displaymath}
 Then $\langle U,V_\lambda\rangle$ is discrete and free on the indicated generators. \qed
\end{lemma}

\begin{example}
  In \Cref{fig15}, we compare the bounds from \cref{ineq} (transformed into $\rho$-coordinates via \cref{rholambda}) and the bounds from \cref{lem2} on the same axes for a
  range of different values of $ p $. The figure clearly shows that when $ p $ is small then we obtain better bounds using $ \lambda$-coordinates. As $ p \to \infty $
  and the isometric circles approach tangency, it is harder and harder for the isometric circles to give an interactive pair and so for large $ p $ the $ \rho$-bounds
  become stronger than the $ \lambda$-bounds.
\end{example}

\begin{figure}
\centering
\includegraphics[width=\textwidth]{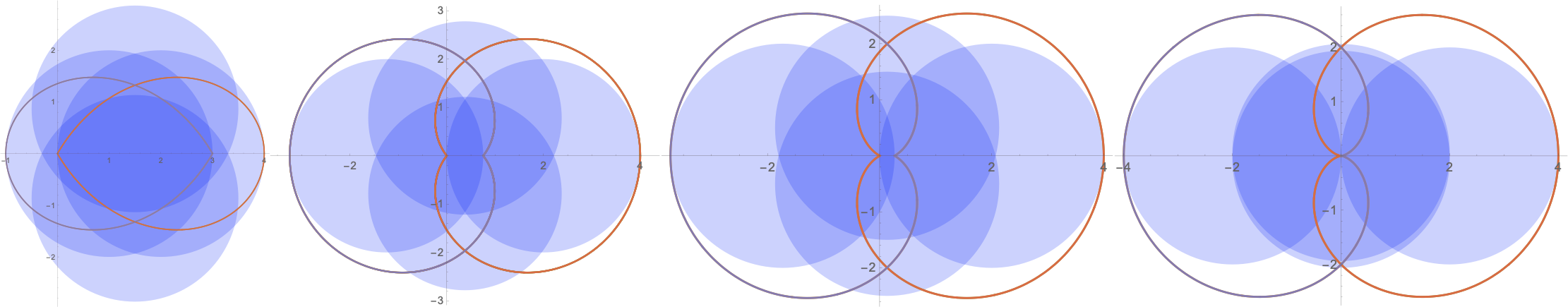}
\caption{Shaded discs from \cref{lem2} compared with the $\lambda$ regions of \cref{ineq}, when $ p = q $. From left : $p=3, 7, 15, 100$.\label{fig15}}
\end{figure}

Motivated by the example $ p = q $, we compare the information in the two representations for the conjugate groups $\langle U,V_\lambda\rangle $ and $\langle A,B_\rho \rangle$.
\begin{theorem}\label{lem:main_res_part_2}
  If $\rho\in \oC \setminus \overline{\mathcal{R}}_{p,q}$, then $  \abs{\Im \rho } <  2  \sqrt{1-\sin ^2 \frac{\pi }{p}  \sin ^2 \frac{\pi }{q} } $. These lines pass through the $ \pm \frac{1}{2} $ cusp
  groups on $ \partial \mathcal{R}_{p,q} $ which were defined in \cref{thmx}.\qed
\end{theorem}

\begin{proof}
  Let $\omega$ be the involution
  \begin{displaymath}
    \omega(z)=\frac{i z+\csc \frac{\pi }{p} -1}{z-i}.
  \end{displaymath}
  This is represented by the matrix
  \begin{displaymath}
  W=\frac{1}{\sqrt{1-2 \sin  \frac{\pi }{p} }}
  \begin{pmatrix}
    \sqrt{\sin \frac{\pi }{p} } & \frac{i \left(\sin \frac{\pi }{p} -1\right)}{\sqrt{\sin \frac{\pi }{p} }} \\
    -i \sqrt{\sin \frac{\pi }{p} } & -\sqrt{\sin \frac{\pi }{p} }
  \end{pmatrix}.
  \end{displaymath}
  One can check that $WUW^{-1} = A$ and that
  \begin{displaymath}
    WVW^{-1} = \sin \frac{\pi}{q}
    \begin{pmatrix}
    \cot  \frac{\pi }{q} +\frac{i  \left((\lambda ^2+1) \sin\frac{\pi }{p}-1\right)}{\lambda  \left(2 \sin\frac{\pi }{p}-1\right)} &
    -\frac{\left(\lambda -\csc\frac{\pi }{p}+1\right) \left(\lambda +\csc\frac{\pi }{p}-1\right)}{\lambda  \left(\csc\frac{\pi }{p}-2\right)} \\
    -\frac{\left(\lambda ^2-1\right) }{\lambda  \left(\csc\frac{\pi }{p}-2\right)} &
    \cot\frac{\pi }{q}-\frac{i \left(\left(\lambda ^2+1\right) \sin\frac{\pi }{p}-1\right)}{\lambda  \left(2 \sin\frac{\pi }{p}-1\right)}
    \end{pmatrix}
  \end{displaymath}
  The mapping $\omega$ preserves the imaginary axis,  and $\omega(\infty)=i\in \mathbf{K}$, so the complement of the isometric discs of $A$ is mapped to $\mathbf{K}$. Thus $A$ and $Y := WVW^{-1} $ form a conic interactive pair. In view of \cref{thm4} we set (using \cref{rholambda} in the third equality)
  \begin{displaymath}
    c_\lambda =Y_{2,1}=- \frac{\left(\lambda ^2-1\right) \sin\frac{\pi }{q}}{\lambda  \left(\csc\frac{\pi }{p}-2\right)} = \frac{  \sqrt{\rho(\rho -4 \sin \frac{\pi }{p}  \sin \frac{\pi }{q}) }}{2 \sin \frac{\pi }{p} -1}.
  \end{displaymath}
  To apply \cref{thm4} we can replace $c_\lambda$ with scalar multiple of it (since the only thing which matters is its direction), and so we set
  \begin{displaymath}
  \nu_\lambda =   \lambda^{-1} \left(\lambda ^2-1\right).
  \end{displaymath}
  If $\lambda$ is imaginary, then so too is $ \nu_\lambda$ and so $ i\nu_\lambda t \in \R $ for all $ t \in \R $. When $ \lambda = is $ for $ s > 0 $, then equality holding in \cref{ineq} implies that
  \begin{displaymath}
    s = \csc \frac{\pi }{p}  \csc \frac{\pi }{q}  + \sqrt{\csc ^2 \frac{\pi }{p}  \csc ^2 \frac{\pi }{q} -1}.
  \end{displaymath}
  Then \cref{rholambda} gives us
  \begin{displaymath}
    \rho_\lambda =-2 \sin \frac{\pi }{p}  \sin \frac{\pi }{q} +2 i \sqrt{1-\sin ^2 \frac{\pi }{p}  \sin ^2 \frac{\pi }{q} }.
  \end{displaymath}
  Hence by \cref{thm4}, the horizontal line $ \{ \rho_\lambda + i\nu_\lambda t : t \in \R \} $ lies in $ \overline{\mathcal{R}}_{p,q} $, as does the half-space which it defines that does not contain $ 0$.
\end{proof}

\section{Faithful representations of the braid group on three strands}

Motivated by the study of $ q$-rationals, Morier-Genoud, Ovsienko, and Veselov \autocite{mgov24} studied the locus of $ \mu \in \C^*$ such that the family of representations
\begin{displaymath}
  \tilde{\rho}_\mu : B_3 \xrightarrow{\rho} \GL(2, \Z[t^{\pm 1}]) \xrightarrow{t\,\mapsto\,\mu} \GL(2,\C)
\end{displaymath}
is faithful, where $ B_3 $ is the 3-strand braid group and $ \rho $ is the reduced Burau representation (defined below). They proved that these specialised Burau representations
are faithful for all $t\in \C^*$ outside the annulus $ 3 - 2\sqrt{2} \leq \abs{\mu} \leq 3 + 2\sqrt{2} $, and they conjectured \autocite[Conjecture 3]{mgov24} that this bound may be improved to the exterior of the annulus
\begin{displaymath}
  \frac{3-\sqrt{5}}{2} \leq \abs{\mu} \leq \frac{3 + \sqrt{5}}{2}.
\end{displaymath}

In this section, we will prove their conjecture and in fact give much tighter bounds on the faithfulness  (\cref{cor:full_bound}).

We first introduce the relevant objects from the theory of braid groups. Let $ B_3 $ denote the 3-strand braid group, and let $ \sigma_1, \sigma_2 $ be the standard Artin generators, so
\begin{displaymath}
  B_3 = \langle \sigma_1,\sigma_2 : \sigma_1 \sigma_2 \sigma_1 = \sigma_2 \sigma_1 \sigma_2 \rangle.
\end{displaymath}

The group $B_3$ admits a representation $ \rho : B_3 \to \GL(2,\Z[t^{\pm 1}]) $ given by
\begin{displaymath}
  \rho(\sigma_1) = \begin{pmatrix} -t & 1 \\ 0 & 1 \end{pmatrix}\;\text{and}\; \rho(\sigma_2) = \begin{pmatrix} 1 & 0 \\ t & -t \end{pmatrix}.
\end{displaymath}
This representation is called the \textit{reduced Burau representation}; by standard results in homological braid theory, it is faithful \autocite[\S 3.3]{kassel}.

For generic $ \mu \in \C^* $ the map $ B_3 \to \GL(2,\C) $ induced by substituting $t$ by $\mu$ is faithful. We write $ \pi_{\mu} : \GL(2,\Z[t^{\pm 1}]) \to \PSL(2,\C) $ for the projection of this map.

The image $ G_\mu $ of $ \pi_\mu \circ \rho $ is generated by
\begin{displaymath}
  A = \frac{1}{\sqrt{-\mu}} \begin{pmatrix} -\mu & 1 \\ 0 & 1 \end{pmatrix}\;\text{and}\;B = \frac{1}{\sqrt{-\mu}} \begin{pmatrix} 1 & 0 \\ \mu & -\mu \end{pmatrix}.
\end{displaymath}

In \autocite{mgov24} it is observed that the projectivisation of this representation is (after substitution of $ -q $ for $ t $) exactly the $q$-deformed
modular group $ \PSL(2,\Z)_q $ of \autocite{mgo20}.

To find a presentation of $ G_\mu $ for generic $\mu$ we only need to identify all the elements
of the image which are scalar multiples of the identity; these elements lie in the centre of $ \GL(2,\C) $
The centre $ Z(B_3) $ of $ B_3 $ is the cyclic group generated by $ (\sigma_1 \sigma_2 \sigma_1)^2 $, which is mapped by $ \rho $ to $ t^3 I_2 $.

Therefore a presentation of $G_\mu$ for generic $\mu$ is given by
\begin{displaymath}
  G_\mu = \langle A, B : ABA = BAB,\, (ABA)^2 = 1 \rangle.
\end{displaymath}
It will be convenient to change to the generators $R = ABA = BAB$, $S = AB$.

From the presentation, we see that $ G_\mu \simeq \PSL(2,\Z) $. When $\mu=-1$ we obtain the usual embedding of $\PSL(2,\Z)$ into $\PSL(2,\C)$. Recall that the quotient $ \Omega(\PSL(2,\Z))/\PSL(2,\Z) $ is a pair of
spheres, each with marked points of order $2$, $3$, and $\infty$; the hyperbolic 3-fold
is accordingly obtained from $S^2 \times (0,1)$ by drilling out two ideal cone arcs (one each of orders
$2$ and $3$, represented by $ R $ and $ S $ respectively) and one rank one cusp (represented by $RS = A^{-1}$). The generic 3-orbifold uniformised by a group
$G_\mu$ is obtained from $ B^3 $ by drilling out two ideal cone arcs, one of order $2$ and one of order $3$ (i.e.\ the rank 1 cusp of the $ \PSL(2,\Z) $ quotient expands to a deleted tube). This can be proved by exhibiting a conjugacy between $G_\mu$ and $\Gamma_\rho$. For this, observe that the conjugacy class of a group in this family is determined by the cross-ratio of the fixed points of the generators
(equality of cross-ratios implies existence of a conjugacy between the two groups). The fixed points of the generators $R$ and $S$ of $ G_\mu $ are
\begin{displaymath}
  \frac{1}{\sqrt{\mu}},\; -\frac{1}{\sqrt{\mu}},\; \frac{1}{2}+\frac{i \sqrt{3}}{2},\; \frac{1}{2}-\frac{i \sqrt{3}}{2}
\end{displaymath}
and the fixed points of the corresponding generators $A$ and $B$ of $ \Gamma_\rho $ are
\begin{displaymath}
  \infty,\; \frac{e^{\pi i/2}}{1-(e^{\pi i/2})^2},\; 0,\; \frac{(e^{\pi i/3})^2 - 1}{e^{\pi i/3} \rho}.
\end{displaymath}
Setting the respective cross ratios to be equal gives the equation
\begin{displaymath}
  1-\frac{2 \sqrt{3} \sqrt{\mu}}{i \mu+\sqrt{3} \sqrt{\mu}-i} = 1-\frac{2 \sqrt{3}}{\rho }
\end{displaymath}
and solving for $ \rho $ gives
\begin{equation}\label{eq:mu_to_rho}
  \rho =  i\sqrt{\mu} + \sqrt{3} - \frac{i}{\sqrt{\mu}}.
\end{equation}

This defines a map $\rho_\mu$ from the deformation space of groups $ G_\mu $ to the deformation space $ \mathcal{R}_{2,3} $. In \cref{fig:slice_in_mu_coords} we show a set of points which fill out
the set of $ \mu $ such that $ G_\mu \not\simeq \Z_2 * \Z_3 $: the complement of these points is the locus of discrete faithful representations.

\begin{remark}
  There are many discrete representations that are not faithful, i.e.\ that appear within the shaded region of \cref{fig:slice_in_mu_coords}. For instance, it is relatively easy to
  classify the discrete representations that lie on the unit circle in the figure. The unit circle in $\mu$-coordinates is the image of the real line in $ \rho$-coordinates, and it is a classic
  result of Knapp~\cite{Knapp} and Brenner~\cite{Brenner} that the discrete groups in the interval $ \mathbb{R} \setminus \mathcal{R}_{2,3} $ are exactly the triangle groups,
  so that $ \tr AB_\rho = 2\cos \theta $ for some $ \theta $ a submultiple of $ \pi $.
  By computing directly with the matrices for $ A $ and $ B_\rho $ we find $ \rho = \sqrt{3} - \tr AB_\rho $. Rearranging \eqref{eq:mu_to_rho} we obtain
  \begin{displaymath}
    0 = i\mu^2 + (\sqrt{3} - \rho) \mu - i = i\mu^2 + (\tr AB_\rho) \mu - i = i\mu^2 + (2\cos \theta) \mu - i
  \end{displaymath}
  Solving for $ \mu $ we obtain
  \begin{displaymath}
    \mu = \frac{-2\cos \theta \pm \sqrt{4\cos^2 \theta - 4}}{2i} = \frac{-\cos \theta \pm i\sin \theta}{i} = i\cos \theta \pm \sin \theta = i e^{\pm i\theta}.
  \end{displaymath}
  Thus the discrete representations of $ B_3 $ on the unit circle are obtained from the roots of unity by quarter-turns. The roots of unity themselves
  are of great interest in the theory of quantum representations, and in this particular context have been studied by Funar and Kohno~\cite{FunarKohno}.
\end{remark}

\begin{figure}\label{fig10}
\centering
\begin{subfigure}[c]{0.48\textwidth}
  \centering
  \includegraphics[height=5cm]{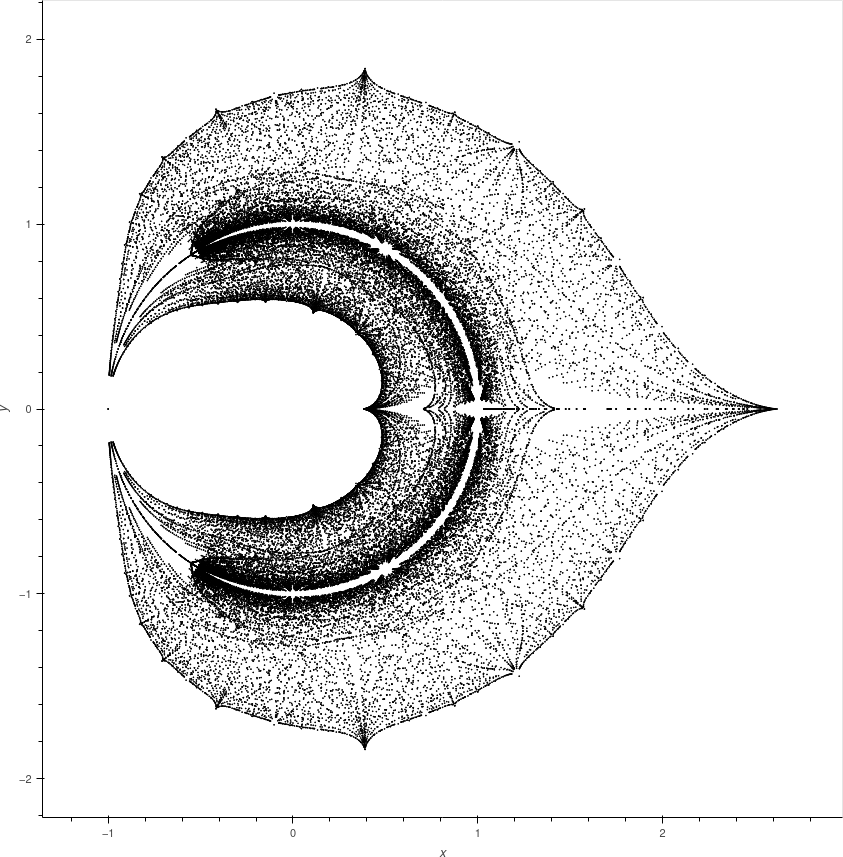}
  \caption{Points plotted approximate $ \IC \setminus \mathcal{R}_{2,3} $ in $\mu$-coordinates.\label{fig:slice_in_mu_coords}}
\end{subfigure}\hspace*{\fill}%
\begin{subfigure}[c]{0.48\textwidth}
  \centering
  \includegraphics[height=5cm]{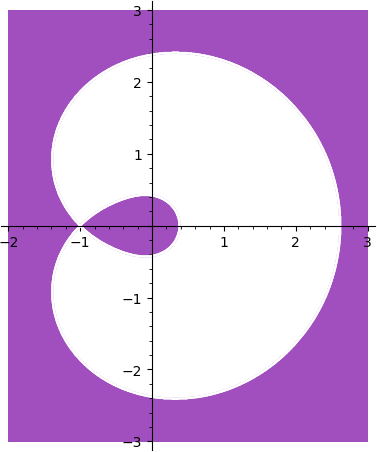}
  \caption{Shaded area consists of points satisfying the bound of \cref{cor:full_bound}.\label{fig:bound}}
\end{subfigure}
\caption{The deformation space $ \mathcal{R}_{2,3}$.}
\end{figure}

\begin{remark}\label{rem:dt}
  Dunfield and Tiozzo \autocite{dt} studied the roots of the Alexander polynomials of various closed 3-braids; these are also quantum invariants
  of knots related to the Burau representations but this time taking the classical braid closure and not the $2$-bridge closure,
  and so one would expect them to be related somehow to the Farey polynomials. We draw the readers attention to the intriguing similarity
  between \cref{fig:slice_in_mu_coords}, and Figure~1.1 of \autocite{dt}; and the similarity between \cref{fig1} above  and Figure~1.3 of \autocite{dt}.
  This suggests that the set of the roots of the Alexander polynomials of all closed 3-braids in the Dunfield--Tiozzo trace coordinates has similar
  properties as a subset of $ \IC $ to the union of the $(-2)$-levelsets of the Farey polynomials for $ \IZ_2 * \IZ_3 $ (this latter set is conjecturally dense
  in the set of $ \rho $ such that the group defined in \cref{XY} is indiscrete and is what is plotted in \cref{fig:slice_in_mu_coords}).
\end{remark}

\begin{theorem}\label{bound}
  For $ \mu \in \C^* $, let $ z=\sqrt{\mu}-1/\sqrt{\mu}$. If
  \begin{displaymath}
  3 \leq \abs{ z \pm \sqrt{z^2 + 3} }
  \end{displaymath}
  then $ G_\mu $ is isomorphic to $ B_3/Z(B_3) $ via the map $ \rho_\mu $.
\end{theorem}
\begin{proof}
  Consider the parameterisation of $ \Z_2 * \Z_3 $ representations studied in \cref{sec:ref}:
  \begin{displaymath}
    \left\langle U = \begin{pmatrix} \frac{1}{2} & -\frac{\sqrt{3}}{2} \\ \frac{\sqrt{3}}{2} & \frac{1}{2} \end{pmatrix},\;%
    V = \begin{pmatrix} 0 & -\lambda  \\ \frac{1}{\lambda} & 0 \end{pmatrix} \right\rangle.
  \end{displaymath}
  Using the cross ratio we may obtain the change of variables
  \begin{displaymath}
    \lambda = \frac{\sqrt{3} {\left(\mu \pm \sqrt{\mu^2 + \mu + 1} - 1\right)}}{3\sqrt{\mu}}.
  \end{displaymath}

  From \cref{ineq}, we see that if $ \abs{\lambda} > \sqrt{3} $, then $ \langle U, V \rangle $ is discrete and is isomorphic to $ \Z_2 * \Z_3 $. Substituting
  the change of variables for $ \lambda $ and using $z = \sqrt{\mu}-1/\sqrt{\mu}$, we obtain the desired inequality.
\end{proof}

In \cref{fig:bound}, we shade the set of $ \mu $ defined by the inequality of \cref{bound} .
This bound is sharp at the image of the $1/2 $ cusp of the Riley slice: one can show that this cusp is the point $ \rho = \sqrt{3} + i $;
its image under $ \mu$-coordinates is exactly the point $ (3+\sqrt{5})/2 $. It is also sharp at the $3/2$ cusp corresponding to $\mu = -1$.

\begin{lemma}\label{hopfian}
If $G_\mu\cong \PSL(2,\Z)$ then $\ker(\pi_\mu \circ \rho)=Z(B_3)$.
\end{lemma}
\begin{proof}
Let $K = \ker(\pi_\mu \circ \rho)$ then $G_\mu=\mathrm{Im}(\pi_\mu \circ \rho)=B_3/K$. Let $Z = Z(B_3)$.
Since $Z \triangleleft K \triangleleft B_3$ we obtain that  $(B_3/Z)/(K/Z) \cong B_3/K = G_\mu$ by the third isomorphism theorem. Since $B_3/Z\cong \PSL(2,\Z)$ and since by assumption $G_\mu \cong \PSL(2,\Z)$ we obtain that $K=Z$ as desired since $\PSL(2,\Z)$ is finitely generated and residually finite \autocite{malcev} hence has no proper quotients isomorphic to itself \autocite[Theorem~IV.4.10]{ls}.
\end{proof}

\begin{corollary}\label{cor:full_bound}
  If $ z = \sqrt{\mu} -1/\sqrt{\mu}$ satisfies the bound in \cref{bound} and $\mu\neq 0,-1$, then $ \tilde{\rho}_\mu $ is faithful.
\end{corollary}
\begin{proof}
  Let $ \tilde{\rho}_\mu : B_3 \to \GL(2,\C) $ be the specialised Burau representation. Suppose $ \mathfrak{a} \in \ker \tilde{\rho}_\mu $. Composition with the projection map yields $ \mathfrak{a}\in \ker(\pi_\mu \circ \rho)$ thus by \cref{hopfian} we must have $ \mathfrak{a} \in Z(B_3) $. Hence $ \mathfrak{a} = (\sigma_1 \sigma_2 \sigma_1)^{2k} $ for some $ k $. But we can explicitly compute,
  \begin{displaymath}
    \tilde{\rho}_{\mu}  (\sigma_1 \sigma_2 \sigma_1)^{2k} = \mu^{3k} \mathrm{Id},
  \end{displaymath}
  hence $ \mu^{3k} = 1 $ and either $ \abs{\mu} = 1 $ or $ k = 0 $. So let $\mu=e^{i2\theta}$ and suppose that $z = 2i\sin(\theta)$ satisfies \cref{bound}. By the triangle inequality for absolute values we obtain $3\leq \abs{2\sin\theta}+\sqrt{\abs{3-4\sin^2(\theta)}}$ hence $\sin \theta = \pm 1$ hence $\mu = -1$.
\end{proof}

\printbibliography

\end{document}